\documentclass[12pt]{article} 
\usepackage{amsmath, amssymb, amsthm, bm, color}

\theoremstyle{definition}
\newtheorem{thm}{Theorem}[section]
\newtheorem{prop}[thm]{Proposition}     

\newtheorem{cor}[thm]{Corollary}
\newtheorem{rem}[thm]{Remark}

\numberwithin{equation}{section}

\newcommand{\R }{{\mathbf R}}
\newcommand{\N }{{\mathbf N}}

\newcommand{\al}{\alpha}
\newcommand{\be}{\beta}
\newcommand{\gm}{\gamma}
\newcommand{\de}{\delta}
\newcommand{\De}{\Delta}
\newcommand{\ep}{\varepsilon}
\newcommand{\lm}{\lambda}

\newcommand{\cF}{{\cal F}}

\newcommand{\supp}{\operatorname{supp}}

\title{Qualitative properties of generalized principal eigenvalues for superquadratic viscous Hamilton-Jacobi equations}

\author{Emmanuel Chasseigne\footnote{Laboratoire de Math\'ematiques et Physique Th\'eorique (UMR CNRS 6083), F\'ed\'eration Denis Poisson (FR CNRS 2964), Universit\'e Fran\c{c}ois Rabelais, Parc de Grandmont, 37200 Tours, France. Email: {\tt emmanuel.chasseigne@lmpt.univ-tours.fr}.}  \ and  Naoyuki Ichihara\footnote{Department of Physics and Mathematics, Aoyama Gakuin University, 5-10-1 Fuchinobe, Chuo-ku, Sagamihara-shi, 
Kanagawa 252-5258, Japan. Email: {\tt ichihara@gem.aoyama.ac.jp}. }}

\date{}
\begin{document}
\maketitle

\begin{abstract}
This paper is concerned with the ergodic problem for superquadratic viscous Hamilton-Jacobi equations with exponent $m>2$. We prove that the generalized principal eigenvalue of the equation converges to a constant as $m\to\infty$, and that the limit coincides with the generalized principal eigenvalue of an ergodic problem with gradient constraint. We also investigate some qualitative properties of the generalized principal eigenvalue with respect to a perturbation of the potential function. It turns out that different situations take place according to $m=2$, $2<m<\infty$, and the limiting case $m=\infty$.

\end{abstract}

\section{Introduction}
In this paper we study the ergodic problem for the following superquadratic viscous Hamilton-Jacobi equation with exponent $m>2$:
\begin{equation}\label{EPm.intro}
\lm-\De u+\frac1m |Du|^m-f=0\quad \text{in }\ \R^N,
\end{equation}
where $Du$ and $\De u$ denote the gradient and the Laplacian of $u:\R^N\to \R$, respectively, and $f:\R^N\to\R$ is assumed to be continuous on $\R^N$ and to vanish as $|x|\to\infty$.  
The unknown of (\ref{EPm.intro}) is the pair of a real constant $\lm$ and a function $u$. We denote by $\lm_m$ the generalized principal eigenvalue of (\ref{EPm.intro}) which is defined by
\begin{equation}\label{deflm}
\lm_m:=\sup\{\lm\in \R\,|\, \text{ (\ref{EPm.intro}) has a continuous viscosity subsolution $u$\,} \}.
\end{equation}
Here and in what follows, unless otherwise specified, every solution (subsolution, supersolution) $u$ is understood in the viscosity sense. We refer, for instance, to \cite{CIL92,Ko04} for the definition and fundamental properties of viscosity solutions. 

The objective of this paper consists of two parts, which we present as A and B below. 

\vskip0.3cm

\noindent\textsc{A. Convergence as $m\to\infty$.}
We study the convergence of $\lm_m$ as $m\to \infty$. More precisely, let us consider the following ergodic problem with gradient constraint:   
\begin{equation}\label{EPinf.intro}
\max\big\{\lm-\Delta u-f,|Du| -1\big\}=0\quad\text{in }\ \R^N.
\end{equation}
Let $\lm_\infty$ denote the generalized principal eigenvalue of (\ref{EPinf.intro}) defined, similarly as (\ref{deflm}), by the supremum of $\lm\in\R$ such that (\ref{EPinf.intro}) has a continuous viscosity subsolution $u$.
Then we prove that $\lm_{m}$ converges to $\lm_\infty$ as $m\to\infty$. In this sense, ergodic problem (\ref{EPinf.intro}) can be regarded as the extreme case of (\ref{EPm.intro}) where $m=\infty$. 
Note that (\ref{EPinf.intro}) has been studied by \cite{Hy12,Hy15} for functions $f$ that are smooth, convex, and of superlinear growth as $|x|\to\infty$. In these papers, $\lm_\infty$ is derived from the limit of $\de v_\de(0)$ as $\de\to 0$, where $v_\de$ is the solution to the following equation:
\begin{equation*}
\max\big\{\de v_\de-\Delta v_\de-f,|Dv_\de| -1\big\}=0\quad\text{in }\ \R^N.
\end{equation*}
The present paper provides another characterization of $\lm_\infty$ in terms of $\lm_m$ under a different type of assumptions on $f$. 

\vskip0.3cm

\noindent\textsc{B. Qualitative properties.}
We introduce a real parameter $\be$ and consider (\ref{EPm.intro}) and (\ref{EPinf.intro}) with $\be f$ in place of $f$. We are interested in qualitative properties of the generalized principal eigenvalue $\lm_m=\lm_{m,\be}$ with respect to $\be$. In order to illustrate our main results briefly, we assume, for a moment, that $f$ is nonnegative in $\R^N$ with compact support (this can be relaxed, see Section 4). Then it turns out that there exists a critical value $\be_c\leq 0$ such that $\lm_{m,\be}=0$
for all $\be\geq \be_c$, while $\lm_{m,\be}<0$ for all $\be<\be_c$. Notice here that the value of $\be_c$, especially, its negativity depends sensitively on $m$ and $N$. More specifically, the following three situations occur according to the choice of $m$:
\begin{itemize}
\item[(a)] if $m=2$, then $\be_c=0$ for $N=1,2$ and $\be_c<0$ for all $N\geq 3$;
\item[(b)] if $2<m<\infty$, then $\be_c=0$ for $N=1$ and $\be_c<0$ for all $N\geq 2$;
\item[(c)] if $m=\infty$, then $\be_c<0$ for all $N\geq 1$. 
\end{itemize}
The quadratic case (a) has been proved in \cite[Theorem 2.5]{Ic13}, and the second claim in (b) (i.e., the case where $2<m<\infty$ and $N\geq 2$) is also suggested by \cite[Theorem 2.4]{Ic15} in a slightly different context. The essential novelty of this paper, compared with \cite{Ic13,Ic15}, lies in the simultaneous derivation of (b) and (c) in combination with the convergence result obtained in part A. In particular, claim (c) for $N\geq 2$ can be derived by passing to the limit of (b) as $m\to\infty$. To the best of our knowledge, such a qualitative analysis of $\lm_{m,\be}$, especially for $m=\infty$, seems to be new. 
We remark that we consider not only nonnegative functions $f$ but also sign-changing ones, which lead to a more complex
picture where two critical parameters $\beta_-\leq\beta_+$ will
play the role of the above $\beta_c$. For instance, if $N\geq2$
and $2<m\leq\infty$, then there exist $\beta_-<0<\beta_+$ such that
$\lambda_{m,\beta}=0$ for any $\beta\in[\beta_-,\beta_+]$, while
$\lambda_{m,\beta}<0$ outside this interval. See Section 4 for details.

\vskip0.3cm

Our study of critical value $\be_c$ is strongly motivated by the stochastic
control interpretation of $\lm_{m,\be}$. Loosely speaking, if $2\leq m<\infty$, then
the principal eigenvalue $\lm_{m,\be}$ coincides with the optimal value of the following
ergodic stochastic control problem: 
\begin{equation}\label{min}
\begin{aligned}
&\text{Minimize}\quad \limsup_{T\to\infty}\frac1T E\left[\int_0^T \Big\{\frac{1}{m^\ast}|\xi_t|^{m^\ast}+\be f(X^\xi_t)\Big\}dt\right],\\
&\text{subject to}\quad X^\xi_t=\sqrt{2}W_t+\int_0^t\xi_sds,\quad t\geq 0,
\end{aligned}
\end{equation}
where $m^\ast:=m/(m-1)$, and $W=(W_t)$ and $\xi=(\xi_t)$ denote, respectively,
an $N$-dimensional standard Brownian motion and an $(\cF_t)$-adapted control
process defined on some filtered probability space $(\Omega,\cF,P;(\cF_t))$. If $f\geq 0$ in $\R^N$ and 
$\be\geq 0$, then this is nothing but a minimization problem of the total cost
$(1/m^\ast)|\xi_t|^{m^\ast}+\be f(X^\xi_t)$. The situation becomes delicate
as far as $\be<0$. Intuitively, the controller of the optimization problem
(\ref{min}) falls into a trade-off situation between minimizing the cost
$(1/m^\ast)|\xi_t|^{m^\ast}$ and maximizing the reward $|\be| f(X^\xi_t)$. The
dominant term depends on the magnitude of $|\be|$, and the critical value
$\be_c$ is determined as the threshold at which the controller changes his/her
optimal choice: either ``minimize cost" or ``maximize reward". 
In particular, the negativity of $\be_c$ implies the existence of such ``phase
transition", which we intend to characterize in the present paper.  

As to the limiting case where $m=\infty$, the value $\lm_{\infty,\be}$ is related to the following
singular ergodic stochastic control problem: \begin{equation*}
\begin{aligned}
&\text{Minimize}\quad \limsup_{T\to\infty}\frac1T E\left[|\eta|_T+\int_0^T \be f(X^\eta_t)\,dt\right],\\
&\text{subject to}\quad X^\eta_t=\sqrt{2}W_t+\eta_t,\quad t\geq 0,
\end{aligned}
\end{equation*}
where $\eta=(\eta_t)$ stands for an $(\cF_t)$-adapted control process of bounded variations, and $|\eta|_T$ denotes its bounded variation norm. We refer, for instance, to \cite{MRT92} and references therein for more information on singular ergodic stochastic control and associated PDEs with gradient constraint. See also \cite{Ic12, Ic13, Ic15}
for the stochastic control interpretation of $\lm_{m,\be}$ for $2\leq m<\infty$. 
In this paper, we focus only on the PDE aspect and do not discuss its probabilistic counterpart.  

The organization of the paper is  as follows. 
In the next section, we discuss the solvability of (\ref{EPm.intro}). Specifically, we prove that, for any $\lm\leq \lm_{m}$, there exists a viscosity solution $u$ of (\ref{EPm.intro}). In Section 3, we prove the convergence
of $\lm_m$ as $m\to\infty$. Section 4 is devoted to qualitative properties of
$\lm_{m,\be}$ with respect to  $\be$.

\section{Solvability of (\ref{EPm.intro})}
We collect some notation used throughout the paper. 
For any $R>0$, $B_R$ stands for the open ball of radius $R$, centered at the origin. 
For given $k\in\N\cup\{0\}$, $\gm\in(0,1]$, and $p\in[1,\infty]$, let
$C^{k,\gm}(\R^N)$ and $W^{k,p}(\R^N)$ denote local H\"{o}lder (or Lipschitz if
$k=0,\gm=1$) spaces and Sobolev
spaces, respectively.  Recall that the H\"older/Lipschitz norm over $B_R$ of
functions in $C^{k,\gm}(\R^N)$ depends on $R$, in general. 
We also denote by $C_c^\infty(\R^N)$ the set of smooth functions with compact support. Finally, let $C_0(\R^N)$ stand for the totality of continuous functions $f\in C(\R^N)$ such that $f(x)\to 0$ as $|x|\to\infty$. 

\vskip0.3cm

Let $m>2$ and consider the ergodic problem
\begin{equation}\label{EP}
\lm-\De u+\frac1m|Du|^m =f\quad \text{in }\ \R^N,\qquad u(0)=0, 
\end{equation}
where the constraint $u(0)=0$ is imposed to avoid the ambiguity of additive constant  with respect to $u$. 
Throughout this paper, we assume without mentioning that $f$ satisfies the following:
\vskip0.2cm
\noindent{\bf (A1)}\quad $f\in C_0(\R^N)$. 
\vskip0.2cm

To begin with, we recall some regularity estimates that will be needed repeatedly. 
\begin{thm}\label{CLP}
Let $\al:=(m-2)/(m-1)$. \\
(i) \ For any $R>0$, there exists a constant $M_R>0$ such that
\begin{equation*}
|u(x)-u(y)|\leq M_R|x-y|^\al,\qquad x,y\in B_R,
\end{equation*}
for any locally bounded upper semicontinuous viscosity subsolution $u$ of (\ref{EP}), where $M_R$ depends on $\max_{B_R}|f-\lm|$, but is independent of any large $m>2$. \\
(ii) \ Suppose that $f\in C^{0,1}(\R^N)$. Then, for any $R>0$, there exists a constant $K_R>0$ such that
\begin{equation*}
|u(x)-u(y)|\leq K_R|x-y|,\qquad x,y\in B_R,
\end{equation*}
for any continuous viscosity solution $u$ of (\ref{EP}), where $K_R$ may depend on the sup-norm and the Lipschitz norm of $f-\lm$ over a larger ball, say $B_{R+1}$, but is independent of any large $m>2$. 
\end{thm}

\begin{proof}
This theorem is a direct consequence of \cite[Theorems 1.1 and 3.1]{CLP10}. Notice here that the gradient term of the equation in \cite{CLP10} is not $(1/m)|Du|^m$ but $|Du|^p$ with $p>2$. However, by a careful reading of their proof, one can see that $M_R$ and $K_R$ can be taken uniformly with respect to any large $m>2$.  
\end{proof}

It is obvious from Theorem \ref{CLP} that any locally bounded upper semicontinuous viscosity subsolution of (\ref{EP}) belongs to $ C^{0,\al}(\R^N)$ with $\al=(m-2)/(m-1)$. Taking this fact into account, one can redefine the generalized principal eigenvalue of (\ref{EP}) by 
\begin{equation}\label{lm}
\lm_m:=\sup\{\lm\in \R\,|\, \text{(\ref{EP}) has a viscosity subsolution } u\in C^{0,\al}(\R^N) \}.
\end{equation}
Note here that $\lm_m\ne -\infty$.  Indeed, $(\lm,u)=(\inf_{\R^N}f,0)$ is a viscosity subsolution of (\ref{EP}), so that $\lm_m\geq \inf_{\R^N}f>-\infty$. 

We first observe a few properties of $\lm_m$ that can be verified by its very definition. In what follows, we often use the notation $\lm_m(f)$ to emphasize the dependence of $\lm_m$ on the function $f$.

\begin{prop}\label{cd}
Let $f,g\in C_0(\R^N)$. We denote by $\lm_m(f), \lm_m(g)$ the associated generalized principal eigenvalues of (\ref{EP}), respectively. Then the following (i)-(iii) hold.\\
(i) \ $f\leq g$ in $\R^N$ implies $\lm_m(f)\leq \lm_m(g)$.\\
(ii) \ $(1-\de)\lm_m(f)+\de \lm_m(g)\leq \lm_m((1-\de)f+\de g)$ for any $\de\in (0,1)$.\\
(iii) \ $\lm_m(f+c)=\lm_m(f)+c$ for any $c\in \R$.
\end{prop}
\begin{proof}
We first show (i). Let $u\in C^{0,\al}(\R^N)$ be a viscosity subsolution of (\ref{EP}) with $f$. Then it is also a viscosity subsolution of  (\ref{EP}) with $g$ in place of $f$. Hence, $\lm_m(f)\leq \lm_m(g)$ by definition. One can also verify (ii) similarly. The validity of (iii) is obvious from the definition of $\lm_m$. Hence, we have completed the proof.
\end{proof}

The following result implies that, if $f\in C^{0,1}(\R^N)$, then ``viscosity subsolution" in the definition of $\lm_m$ can be replaced by ``classical subsolution".

\begin{prop}\label{Cinfty}
Suppose that $f\in C^{0,1}(\R^N)$. Then, for any $\lm< \lm_m$, there exists a classical subsolution $u\in C^{\infty}(\R^N)$ of (\ref{EP}). 
\end{prop}

\begin{proof}
Fix any $\lm<\lm_m$ and construct a smooth subsolution $u$ of (\ref{EP}). To this end, we follow the ingenious idea due to \cite{BJ02,Kr00}. Set $f_\ep(x):=\min_{|e|<\ep}f(x+e)$ for $\ep>0$. Then, $f_\ep\in C^{0,1}(\R^N)\cap C_0(\R^N)$, $f_\ep\leq f$ in $\R^N$, and $\{f_\ep\}$ converges to $f$ uniformly in $\R^N$ as $\ep\to 0$. Let $\lm_m^{(\ep)}$ be the generalized principal eigenvalue of (\ref{EP}) with $f_\ep$ in place of $f$.
Then, in view of Proposition \ref{cd} and by choosing $\ep>0$ sufficiently small, we may assume that $\lm<\lm_m^{(\ep)}\leq \lm_m$. In particular, for the above $\lm$, there exists a viscosity subsolution $u^{(\ep)}\in C^{0,\al}(\R^N)$ of (\ref{EP}) with $f_\ep$ in place of $f$. 
Since $f_\ep(\,\cdot\,-e)\leq f$ in $\R^N$ for any $|e|<\ep$, one can also see that $u^{(\ep)}(\,\cdot\,-e)$ is a viscosity subsolution of (\ref{EP}) for any $|e|<\ep$. 

Now, let $\{\rho_\de\}_{\de>0}\subset C_c^\infty(\R^N)$ be a family of mollifier functions, i.e., $\rho_\de\geq 0$ in $\R^N$, $\int_{\R^N}\rho_\de(x)\,dx=1$, and $\supp \rho_\de\subset B_\de$ for all $\de>0$. Set $u_\de^{(\ep)}(x):=(u^{(\ep)}\ast \rho_\de)(x)$ for $\de<\ep$, where $\ast$ stands for the usual convolution. Then, by noting  the convexity of $p\mapsto (1/m)|p|^m$, one can see, similarly as in the proof of \cite[Lemma 2.7]{BJ02}, that $u:=u_\de^{(\ep)}$ is a smooth viscosity subsolution of (\ref{EP}). Since a smooth viscosity subsolution is a classical subsolution, we have completed the proof.  
\end{proof}

We next verify that $\lm_m$ is nonpositive.
\begin{prop}\label{upp}
One has $\lm_m\leq 0$. In particular, $\lm_m$ is finite.
\end{prop}

\begin{proof}
In view of Proposition \ref{cd} (i), it suffices to consider the case where $f\in C^{0,1}(\R^N)$. Fix any $\lm<\lm_m$, and let $u\in C^{\infty}(\R^N)$ be a classical subsolution of (\ref{EP}). Existence of such $u$ is guaranteed by virtue of Proposition \ref{Cinfty}. Then, for any nonnegative test function $\eta\in C_c^\infty(\R^N)$ such that $\int_{\R^N}\eta(x)^{m^\ast}dx=1$, where $m^\ast:=m/(m-1)$,
we have
\begin{align*}
\lm\int_{\R^N}\eta^{m^\ast}dx+\int_{\R^N}Du\cdot D(\eta^{m^\ast})\, dx+\frac1m\int_{\R^N}|Du|^m\eta^{m^\ast}\, dx\leq  \int_{\R^N}f\eta^{m^\ast}\, dx.
\end{align*}
Noting 
$D(\eta^{m^\ast})=\eta^{m^\ast/m}D\eta$ and 
\begin{equation*}
Du\cdot D(\eta^{m^\ast})\leq \frac1m|Du|^m\eta^{m^\ast}+\frac{1}{m^\ast}|D\eta|^{m^\ast},
\end{equation*}
we see that, for any $\ep>0$, 
\begin{equation*}
\lm=\lm\int_{\R^N}\eta^{m^\ast}dx\leq \ep+\int_{\R^N}(f-\ep)_+\eta^{m^\ast}\, dx+\frac{1}{m^\ast}\int_{\R^N}|D\eta|^{m^\ast}dx,
\end{equation*}
where $(f-\ep)_+$ denotes the positive part of $f-\ep$. Furthermore, if we define $\eta_\de(x):=\de^{N/m^\ast}\eta(\de x)$ for $\de>0$, which still satisfies $\int_{\R^N}\eta_\de(x)^{m^\ast}dx=1$ for any $\de>0$, then plugging this into the above $\eta$, we have
\begin{equation}\label{mg}
\lm\leq \ep+\int_{\R^N}(f(x)-\ep)_+\eta_\de(x)^{m^\ast}\, dx+\frac{\de^{m^\ast}}{m^\ast}\int_{\R^N}|D\eta(x)|^{m^\ast}dx.
\end{equation}
Sending $\de\to 0$, we obtain $\lm\leq \ep$. Since $\ep>0$ and $\lm<\lm_m$ are arbitrary, we conclude that $\lm_m\leq 0$. Hence, we have completed the proof. 
\end{proof}

The following proposition states a stability of $\lm_m(f)$ with respect to $f$.  
\begin{prop}\label{staf}
Let $f,g\in C_0(\R^N)$. Then $|\lm_m(f)-\lm_m(g)|\leq \max_{\R^N}|f-g|$.
In particular, if $\{f_n\}\subset C_0(\R^N)$ converges as $n\to\infty$ to some $f\in C_0(\R^N)$ uniformly in $\R^N$, then $\lm_m(f_n)$ converges to $\lm_m(f)$ as $n\to\infty$. Moreover, if $\{u_n\}$ is a family of viscosity solutions of (\ref{EP}) with $f=f_n$ and $\lm=\lm_m(f_n)$, then, along a suitable subsequence, $\{u_n\}$ converges as $n\to\infty$ to a viscosity solution $u$ of (\ref{EP}) with $\lm=\lm_m(f)$ locally uniformly in $\R^N$. 
\end{prop}

\begin{proof}
Since $f\leq g+\max_{\R^N}(f-g)_+$ in $\R^N$, we see, in view of Proposition \ref{cd} (i) and (iii), that $\lm_m(f)-\lm_m(g)\leq \max_{\R^N}(f-g)_+$. Changing the role of $f$ and $g$, we obtain the first claim. The second claim is obvious from the first one. In order to verify the last claim, we observe from Theorem \ref{CLP} that $\{u_n\}$ pre-compact in $C(\R^N)$. Applying Ascoli-Arzela theorem, we see that $\{u_n\}$ converges, along a suitable subsequence, to a function $u\in C^{0,\al}(\R^N)$ locally uniformly in $\R^N$. By the stability property of viscosity solutions, we conclude that $u$ is a viscosity solution of (\ref{EP}) with $\lm=\lm_m(f)$. Hence, we have completed the proof.
\end{proof}

We now state the main result of this section.

\begin{thm}\label{exist}
For any $\lm\leq  \lm_m$, there exists a viscosity solution $u\in C^{0,\al}(\R^N)$ of (\ref{EP}).  Moreover, if $f\in C^{0,1}(\R^N)$, then for any $\lm\leq \lm_m$, there exists a classical solution $u\in C^{2}(\R^N)$ of (\ref{EP}). 
\end{thm}

\begin{proof}
We first prove the latter claim. Let $f\in C^{0,1}(\R^N)$ and fix any $\lm<\lm_m$. Then, by virtue of Proposition \ref{Cinfty}, there exists a classical subsolution $u_-\in C^\infty(\R^N)$ of (\ref{EP}).  
Fix any $R>0$ and consider the Dirichlet problem
\begin{equation}\label{D}
\lm-\De u+\dfrac1m|Du|^m -f=0\quad \text{in }\ B_R,\qquad u=u_-\quad \text{on }\ \partial B_R,
\end{equation}
where $\partial B_R:=\{x\in\R^N\,|\, |x|=R\}$.
Then it is known (e.g. \cite[Th\'{e}or\`{e}me I.1]{Li80}) that there exists a unique classical solution $u_R\in C^{2,\gm}(\overline{B}_R)$ of (\ref{D}) for some $\gm\in(0,1)$. By virtue of Theorem \ref{CLP} together with the Schauder estimate, we see that $\{u_R-u_R(0)\}_{R>0}$ is pre-compact in $C^2(\R^N)$. In particular, letting $R\to \infty$ along a suitable subsequence $\{R_j\}$ if necessary, we see that $\{u_{R_j}\}$ and their first and second derivatives converge as $j\to\infty$ to a function $u\in C^2(\R^N)$ and its corresponding derivatives, respectively, locally uniformly in $\R^N$, and that $u$ is a classical solution of (\ref{EP}). In order to see that (\ref{EP}) with $\lm=\lm_m$ has a classical solution, we choose any sequence $\{\lm^{(n)}\}$ such that $\lm^{(n)}\to \lm_m$ as $n\to\infty$, and let $u^{(n)}$ denote the associated classical solution to (\ref{EP}) with $\lm=\lm^{(n)}$. Then one can see, similarly as above, that $\{u^{(n)}-u^{(n)}(0)\}$ is pre-compact in $C^2(\R^N)$. Passing to the limit as $n\to\infty$ along a suitable subsequence if necessary, we conclude that (\ref{EP}) with $\lm=\lm_m$ has a classical solution.    

We now prove the former claim. Fix any $f\in C_0(\R^N)$ and choose a sequence $\{f_n\}\subset C^{\infty}(\R^N)\cap C_0(\R^N)$ which converges as $n\to\infty$ to $f$ uniformly in $\R^N$. Let $\lm^{(n)}$ be the generalized principal eigenvalue of (\ref{EP}) with $f_n$ in place of $f$. Then, in view of Proposition \ref{staf}, we observe that $\lm^{(n)}\to \lm_m$ as $n\to\infty$. Now, fix any $\lm<\lm_m$. We may assume without loss of generality that $\lm<\lm^{(n)}$ for any $n\geq 1$. For each $n\geq 1$, let $u^{(n)}\in C^2(\R^N)$ denote a classical solution of (\ref{EP}) with $f_n$ in place of $f$. Then, by Theorem \ref{CLP} and the stability of viscosity solutions, we conclude that, along a suitable subsequence,  $\{u^{(n)}\}$ converges as $n\to\infty$ to  a viscosity solution $u\in C^{0,\al}(\R^N)$ of (\ref{EP}) locally uniformly in $\R^N$. We can also construct a viscosity solution of (\ref{EP}) with $\lm=\lm_m$ similarly as in the previous case. Hence, we have completed the proof. 
\end{proof}

Theorem \ref{exist} implies that the following representation formula for $\lm_m$ holds:
\begin{equation*}
\lm_m=\max\{\lm\in \R\,|\, \text{(\ref{EP}) has a viscosity solution } u\in C^{0,\al} (\R^N) \}.
\end{equation*}
Furthermore, if $f\in C^{0,1}(\R^N)$, then 
\begin{align*}
\lm_m=\max\{\lm\in \R\,|\, \text{(\ref{EP}) has a classical solution } u\in C^{2} (\R^N) \}.
\end{align*}

\section{Convergence as $m\to\infty$}

This section is devoted to the convergence of $\lm_m$ as $m\to\infty$. To be precise, we rewrite the limiting equation
\begin{equation}\label{EPinf}
\max\big\{\lm-\Delta u-f,|Du| -1\big\}=0\quad\text{in }\ \R^N,\qquad u(0)=0,
\end{equation}
and redefine the generalized principal eigenvalue of (\ref{EPinf}) by  
\begin{equation}\label{lminf}
\lm_\infty:=\sup\{\lm\in \R\,|\, \text{(\ref{EPinf}) has a viscosity subsolution } u\in C^{0,1}(\R^N) \}.
\end{equation}

The following result is crucial to our convergence result.
\begin{prop}\label{sta}
Let $\{m_k\}\subset \R$ be an increasing sequence such that $m_k\to\infty $ as $k\to\infty$. Let $(\lm_{m_k},u_k)$ be a solution of (\ref{EP}) with $m=m_k$ for each $k$. Suppose that $\lambda_k$ converges to some $\lm\in\R$ as $k\to\infty$. Then, up to a subsequence, $\{u_k\}$ converges  as $k\to\infty$ to a function $u\in C^{0,1}(\R^N)$ locally uniformly in $\R^N$. Moreover, $(\lm,u)$ is a solution of (\ref{EPinf}). 
\end{prop}

\begin{proof}
In view of Theorem \ref{CLP} (i), we see that there exist a subsequence of  $\{u_k\}$, which we denote by $\{u_k\}$ again, and a function $u\in C(\R^N)$ with $u(0)=0$ such that $u_k\to u$ as $k\to \infty$ locally uniformly in $\R^N$. Note that $u\in C^{0,1}(\R^N)$ since the constant $M_R$ in Theorem \ref{CLP} (i) does not depend on any large $m>2$. 

We now verify that $u$ is a viscosity solution of (\ref{EPinf}). We first prove the subsolution property. Fix any $x_0\in \R^N$ and let $\phi\in C^2(\R^N)$ be any function such that $\max_{\R^N}(u-\phi)=(u-\phi)(x_0)$. As is standard, one can assume that the maximum is strict, so that there exists a sequence $\{x_k\}\subset \R^N$ such that $u_k-\phi$ attains its local  maximum at $x_k$ and $x_k\to x_0$ as $k\to\infty$. Then, by the subsolution property of $u_k$, we see that 
\begin{equation}\label{eqm}
\lm_{m_k}-\De\phi(x_k)+\frac{1}{m_k}|D\phi(x_k)|^{m_k}-f(x_k)\leq 0.
\end{equation}
We now suppose that $|D\phi(x_0)|>1$. Then there exists an $\eta>0$ such that $|D\phi(x_k)|\geq 1+\eta$ for all sufficiently large $k$. In particular, we have
\begin{equation*}
\frac{1}{m_k}(1+\eta)^{m_k}\leq -\lm_{m_k}+\De\phi(x_k)+f(x_k).
\end{equation*}
Sending $k\to\infty$, we get a contradiction since the right-hand side remains bounded, whereas the left-hand side goes to infinity as $k\to\infty$. Hence, we have $|D\phi(x_0)|\leq 1$. Furthermore, letting $k\to\infty$ in (\ref{eqm}), we conclude that $\lm-\De\phi(x_0)-f(x_0)\leq0$, which implies that $u$ is a viscosity subsolution of (\ref{EPinf}).

We next prove the supersolution property. Fix any $x_0\in \R^N$ and let $\psi\in C^2(\R^N)$ be such that $\min_{\R^N}(u-\psi)=(u-\psi)(x_0)$. If $|D\psi(x_0)|\geq1$, then there is nothing to prove, so we assume that $|D\psi(x_0)|<1$. In particular, there exists some $\eta>0$ such that $|D\psi(x_k)|\leq 1-\eta$ for all sufficiently large $k$. Furthermore, there exists a sequence $\{x_k\}\subset \R^N$ such that $u_k-\psi$ attains its local minimum at $x_k$ and $x_k\to x_0$ as $k\to\infty$. Then, by the supersolution property of $u_k$, we have
\begin{equation*}
\lm_{m_k}-\De\psi(x_k)+\frac{1}{m_k}|D\psi(x_k)|^{m_k}-f(x_k)\geq 0.
\end{equation*}
Letting $k\to\infty$ in the above inequality, we obtain $\lm-\De\psi(x_0)-f(x_0)\geq0$. Hence, we conclude that $u$ is a viscosity supersolution of (\ref{EPinf}).
\end{proof}

We are now in position to state the main result of this section.
\begin{thm}\label{conv}
Let $\lm_m$ and $\lm_\infty$ be the generalized principal eigenvalues of (\ref{EP}) and (\ref{EPinf}), respectively. Then, $\lm_m$ converges to $\lm_\infty$ as $m\to\infty$. Moreover, equation (\ref{EPinf}) with $\lm=\lm_\infty$ has a viscosity solution $u\in C^{0,1}(\R^N)$.
\end{thm}
\begin{proof}
Set $\overline{\lm}:=\limsup_{m\to\infty}\lm_m$. 
Note that $\overline{\lm}\leq 0$ in view of Proposition \ref{upp}. Let $(\lm_{m_k},u_{m_k})$ be a sequence of solutions to (\ref{EP}) with $m=m_k$ such that $\lm_{m_k}\to \overline{\lm}$ as $k\to\infty$. Then, by taking a subsequence if necessary, we see from  Proposition \ref{sta} that $\{u_{m_k}\}$ converges to a viscosity solution $u\in C^{0,1}(\R^N)$ of (\ref{EPinf}) locally uniformly in $\R^N$. In particular, $\overline{\lm}\leq \lm_\infty$.

To prove the reverse inequality, we set $\underline{\lm}:=\liminf_{m\to\infty}\lm_m$. Fix any $\ep>0$ and let $u\in C^{0,1}(\R^N)$ be a viscosity subsolution of (\ref{EPinf}) with $\lm=\lm_\infty-\ep$. Then, noting that $|Du|\leq 1$ in $\R^N$ in the viscosity sense, we see that, for any $m>2$, $u$ is a viscosity subsolution of 
\begin{equation*}
\lm_\infty-\ep-\frac1m-\De u+\frac1m|Du|^m-f\leq 0\quad \text{in }\R^N.
\end{equation*}
This implies $\lm_\infty-\ep-1/m\leq \lm_m$ for any $m>2$, so that $\lm_\infty-\ep\leq \underline{\lm}$. Since $\ep>0$ is arbitrary, we obtain $\lm_\infty\leq \underline{\lm}\leq \overline{\lm}\leq \lm_\infty$. Hence, we have completed the proof. 
\end{proof}

The next result states that Proposition \ref{Cinfty} remains valid for $m=\infty$.

\begin{prop}\label{C11}
Suppose that $f\in C^{0,1}(\R^N)$. Then, for any $\lm< \lm_\infty$, there exists a classical subsolution $u\in C^{\infty}(\R^N)$ of (\ref{EPinf}). In particular, 
\begin{equation*}
\lm_{\infty}=\sup\{\lm\in \R\,|\, \text{(\ref{EPinf}) has a classical subsolution $u\in C^{\infty}(\R^N)$} \}.
\end{equation*}
\end{prop}

\begin{proof}
Fix any $\lm<\lm_{\infty}$, and let $\{\rho_\de\}_{\de>0}\subset C^\infty_c(\R^N)$ be  such that $\rho_\de\geq 0$ in $\R^N$,  $\int_{\R^N}\rho_\de(x)dx=1$, and $\supp\rho_\de\subset B_\de$ for all $\de>0$. Let $\{\lm_{m_k}\}$ be a sequence of generalized principal eigenvalues of (\ref{EP}) with $m=m_k$ such that $\lm_{m_k}\to \lm_{\infty}$ as $k\to\infty$. Such a sequence exists by virtue of Theorem \ref{conv}. In what follows, we assume that $\lm<\lm_{m_k}$ for all $k\geq 1$. For each $k\geq 1$, let $u^{(k)}\in C^2(\R^N)$ be a classical solution of (\ref{EP}) with $m=m_k$ (and the common $\lm$). Taking a subsequence if necessary, one may also assume that $\{u^{(k)}\}$ converges  as $k\to\infty$ to a viscosity solution $u\in C^{0,1}(\R^N)$ of (\ref{EPinf}) locally uniformly in $\R^N$.  

Now we set $u^{(k)}_{\de}:=u^{(k)}\ast \rho_\de$, $u_{\de}:=u\ast \rho_\de$,  and $f_{\de}:=f\ast \rho_\de$, where $\ast$ stands for the usual convolution. We choose $\de>0$ so small that $\sup_{\R^N}| f_\de- f|<\lm_{m_k}-\lm$ for all $k\geq 1$. Then, since $u^{(k)}$ is a classical solution of (\ref{EP}) with $m=m_k$, we see that $u^{(k)}_{\de}$ enjoys the inequality
\begin{equation*}
\lm-\De u^{(k)}_{\de}+\frac{1}{m_k}|Du^{(k)}_{\de}|^{m_k}-f\leq 0 \quad \text{in }\ \R^N
\end{equation*}
for all $k\geq 1$ and any sufficiently small $\de>0$. 
This implies that  $u^{(k)}_{\de}$ is also a classical subsolution of 
\begin{equation*}
\lm-\De u^{(k)}_{\de}-f\leq 0 \quad \text{in }\ \R^N.
\end{equation*}
Letting $k\to\infty$ and noting the stability of viscosity solutions, we conclude that $u_{\de}$ is a smooth viscosity subsolution, and therefore, a classical subsolution of the same equation. On the other hand, since $|u|\leq 1$ a.e. in $\R^N$, which can be verified as in the proof of Proposition \ref{sta}, we see that $|u_\de|\leq 1$ in $\R^N$. Hence, $u_{\de}$ enjoys (\ref{EPinf}) at every point $x\in\R^N$, and we have completed the proof. 
\end{proof}

\begin{rem}
The first claim of Theorem \ref{exist} remains true for $m=\infty$. Namely, for any $\lm\leq \lm_\infty$, there exists a viscosity solution $u\in C^{0,1}(\R^N)$ of (\ref{EPinf}). To see this, fix any $\lm<\lm_\infty$ and choose an $m_0$ so large that $\lm_m>\lm$ for any $m>m_0$. Let $u_m$, for $m>m_0$, be a viscosity solution of (\ref{EP}). Then, by Proposition \ref{sta}, we conclude that, along a subsequence, $\{u_m\}$ converges to a viscosity solution $u\in C^{0,1}(\R^N)$ of (\ref{EPinf}). The existence of a viscosity solution $u$ to (\ref{EPinf}) with $\lm=\lm_\infty$ has been proved in Theorem \ref{conv}. Hence, the first claim of Theorem \ref{exist} is also valid for $m=\infty$. We do not know if the second claim remains true for $m=\infty$. 
   
\end{rem}

\section{Qualitative properties}  
In this section, we introduce real parameter $\be$ and consider the ergodic problem for $m>2$:
\begin{equation}\label{EPmb}
\lm-\De u+\frac1m |Du|^m-\be f=0\quad \text{in }\ \R^N,\quad u(0)=0,
\end{equation}
and its limiting equation as $m\to\infty$:
\begin{equation}\label{EPinfb}
\max\{\lm-\De u-\be f,|Du|-1\}=0\quad \text{in }\ \R^N,\quad u(0)=0.
\end{equation}
In the rest of this paper, we impose the following assumption on $f$ in addition to (A1): 
\vskip0.2cm
\noindent
{\bf (A2)} \ $f\not\equiv0$ and $|f(x)|\leq C_0\langle x\rangle^{-m^\ast}$ in $\R^N$ for some $C_0>0$, where $\langle x\rangle:=(1+|x|^2)^{1/2}$ and $m^\ast:=m/(m-1)$ with the convention that $m^\ast:=1$ for $m=\infty$.
\vskip0.2cm
 
Let $\lm_{m,\be}$ and $\lm_{\infty,\be}$ be the generalized principal eigenvalues of (\ref{EPmb}) and (\ref{EPinfb}), respectively. In view of Proposition \ref{upp} and Theorem \ref{conv}, we observe that $\lm_{m,\be}\leq 0$ for any $\be\in\R$ and $2<m\leq \infty$. It is also easy to see that $\lm_{m,0}=0$  for any $2<m\leq \infty$. Furthermore, we have the following.

\begin{prop}\label{betac}
Let $2<m\leq \infty$. If $f_-\not\equiv0$, then $\lm_{m,\be}\to -\infty$ as $\be\to\infty$, and if $f_-\equiv0$, then $\lm_{m,\be}=0$ for any $\be>0$. Symmetrically, if $f_+\not\equiv0$, then $\lm_{m,\be}\to -\infty$ as $\be\to-\infty$, and if $f_+\equiv0$, then $\lm_{m,\be}=0$ for any $\be<0$. 
\end{prop}

\begin{proof}
We first consider the case where $2<m<\infty$. In view of Proposition \ref{staf}, we may assume that $f\in C^{0,1}(\R^N)$. Suppose that $f_-\not\equiv0$, and choose any $\eta\in C_c^\infty(\R^N)$ such that $\eta\geq 0$ in $\R^N$, $\int_{\R^N}\eta(x)^{m^\ast}dx=1$, and $\supp \eta\subset \supp f_-$. Then, taking a classical solution $u\in C^{2}(\R^N)$ of (\ref{EPmb}) with $\lm=\lm_{m,\be}$, multiplying both sides of (\ref{EPmb}) by $\eta$, and applying integration by parts, we see as in the proof of Proposition \ref{upp} that

\begin{equation}\label{m}
\lm_{m,\be}\leq -\be \int_{\R^N}f_-(x)\eta(x)^{m^\ast}\, dx+\frac{1}{m^\ast}\int_{\R^N}|D\eta(x)|^{m^\ast}dx.
\end{equation}
Since the integral of $f_-\eta^{m^\ast} $ over $\R^N$ is strictly positive, we conclude that $\lm_{m,\be}\to -\infty$ as $\be\to \infty$. 
We now take the limit as $m\to\infty$ in (\ref{m}). Then, since $m^\ast\to 1$ as $m\to\infty$, we see from Theorem \ref{conv} that the claim is also valid for $m=\infty$. 

We now suppose that $f_-\equiv0$. Then, for any $\be>0$, the pair $(\lm,u)=(0,0)$ is a subsolution of (\ref{EPmb}) and (\ref{EPinfb}). This implies that $\lm_{m,\be}=0$ for any $2<m\leq \infty$ and $\be>0$. 
By choosing $-f$ and $-\be$ in place of $f$ and $\be$, respectively, we see that the latter claim of this proposition is also valid. Hence, we have completed the proof.
\end{proof}

From Propositions \ref{cd} (ii), \ref{upp}, and \ref{betac}, for each $2<m\leq \infty$, one can define $\be_-, \be_+$ by
\begin{equation*}
\be_+:=\max\{\be\in\R\,|\, \lm_{m,\be}=0\},\qquad
\be_-:=\min\{\be\in\R\,|\, \lm_{m,\be}=0\}.
\end{equation*}
Obviously, $-\infty\leq \be_-\leq 0\leq \be_+\leq\infty$, and $\be_+$ (resp. $\be_-$) is finite if and only if $f_-\not\equiv0$ (resp. $f_+\not\equiv 0$). Moreover, since $f\not \equiv 0$, either $\be_+$ or $\be_-$ is finite. 
As is mentioned in the introduction, we wish to know whether $0<|\be_\pm| \ (<\infty)$.  The main result of this section can be stated as follows.
\begin{thm}\label{main2}
Let $\be_+$ be defined as above, and let $f_-\not\equiv0$. \\
(i) \ Suppose that $N\geq 2$ and $2<m\leq \infty$. Then $\be_+>0$.\\
(ii) \ Suppose that $N=1$ and $2<m< \infty$. Then $\be_+=0$. \\
(iii) \ Suppose that $N=1$ and $m=\infty$. Then $\be_+>0$ provided $f_-\in L^1(\R)$.
\end{thm}

Changing $(\beta,f)$ into $(-\beta,-f)$, one has the following symmetrical result as a corollary of Theorem \ref{main2}.

\begin{cor}
Let $\be_-$ be defined as above, and let $f_+\not\equiv0$. \\
(i) \ Suppose that $N\geq 2$ and $2<m\leq \infty$. Then $\be_-<0$.\\
(ii) \ Suppose that $N=1$ and $2<m< \infty$. Then $\be_-=0$. \\
(iii) \ Suppose that $N=1$ and $m=\infty$. Then $\be_-<0$ provided $f_+\in L^1(\R)$.
\end{cor}

\begin{rem}\label{rem:beta.plus.minus} 
If $N\geq 2$ and $f$ is sign-changing, then $\beta_-<0<\beta_+$ for any $2<m\leq \infty$. From the ergodic stochastic control point of view, this implies that there exist two different critical points $\be_+$ and $\be_-$ at which the controller changes his/her optimal strategy. We remark that, if $f$ is nonnegative or nonpositive in $\R^N$, then there is only one such critical point. 
\end{rem}

In the rest of this section, we prove (i)-(iii) of Theorem  \ref{main2} one by one. The key to the proof of claim (i) is the following estimate.  
\begin{prop}\label{be0}
Let $N\geq 2$ and $2<m<\infty$. Set 
\begin{equation*}
\be_0:=\frac{(N-m^\ast)^{m^\ast}}{m^\ast C_0}>0,
\end{equation*}
where $m^\ast:=m/(m-1)$ and $C_0>0$ is the constant in (A2). Then, for any $|\be|\leq \be_0$, there exists a subsolution $u\in C^\infty(\R^N)$ of (\ref{EPmb}) with $\lm=0$. 
\end{prop}

\begin{proof}\label{phi0}
We define $u:\R^N\to\R$ by $u(x):=(K/\al)\langle x\rangle^{\al}$, where $\al=(m-2)/(m-1)$ and $K>0$ is some constant that will be specified later. Then, by direct computations, we see that $Du=K\langle x\rangle^{-m^\ast}x$ and $\De u=KN\langle x\rangle^{-m^\ast}-Km^\ast\langle x\rangle^{-m^\ast-2}|x|^2$. Thus, 
\begin{align*}
-\De u+\frac1m|Du|^m
&=\langle x\rangle^{-m^\ast}\Big\{-KN+Km^\ast|x|^2\langle x\rangle^{-2}\Big\}+\frac{K^m}{m}|x|^m\langle x\rangle^{-mm^\ast}\\
&=\langle x\rangle^{-m^\ast}\Big\{-KN+Km^\ast|x|^2\langle x\rangle^{-2}+\frac{K^m}{m}|x|^m\langle x\rangle^{-m}\Big\}\\
&\leq \langle x\rangle^{-m^\ast}\Big\{-(N-m^\ast)K+\frac{K^m}{m}\Big\}.
\end{align*}
Since the function $K\mapsto f(K):=(K^m/m)-(N-m^\ast)K$ attains its minimum $-(1/m^\ast)(N-m^\ast)^{m^\ast}$ at $K=(N-m^\ast)^{1/(m-1)}=:K_m$, we choose $K=K_m$ in the definition of $u$ to obtain
\begin{equation*}
-\De u+\frac1m|Du|^m+\be f\leq \langle x\rangle^{-m^\ast}\Big\{ |\be| C_0-\frac{1}{m^\ast}(N-m^\ast)^{m^\ast}\Big\}\quad \text{in }\ \R^N.
\end{equation*}
This implies that $u$ is a subsolution of (\ref{EP}) with $\lm=0$ provided $|\be|\leq \be_0$. Hence, we have completed the proof.
\end{proof}

As a corollary of this proposition, one can prove claim (i) of Theorem \ref{main2}.
\begin{proof}[Proof of Theorem \ref{main2} (i)]
Let $\be_0$ be the constant taken from Proposition \ref{be0}. Then, it is obvious that $\be_+\geq \be_0>0$ for any $2<m<\infty$. Moreover, since $m^\ast\to 1$ as $m\to\infty$, we see that $\be_+\geq \be_0\geq (N-1)/(2C_0)>0$ for any large $m$. Hence, letting $m\to\infty$ and noting that $\lm_{m,\be}$ converges to $\lm_{\infty,\be}$  as $m\to\infty$ for any $\be\in\R$, we conclude that $\lm_{\infty,\be}=0$ for any $\be\leq (N-1)/(2C_0)$. This yields that $\be_+>0$ for $N\geq 2$ and $m=\infty$. Hence, we have completed the proof.
\end{proof}

\begin{rem}\label{remIc15}
In the case where $N\geq 2$ and $2<m<\infty$, the positivity $\be_+>0$ has been observed in \cite[Proposition 2.4]{Ic15} when $f\in C^{0,1}(\R^N)$. The new ingredient here is that we have an explicit lower bound of $\be_+$, uniform in $m$, which leads to the positivity of $\be_+$ not only for $2<m<\infty$ but also for $m=\infty$. Recall that $\be_+=0$ for $N=m=2$ (see \cite{Ic13}). This exhibits a striking contrast between quadratic and superquadratic cases. 
\end{rem}

In what follows, we concentrate on the case where $N=1$, in which case the ergodic problem (\ref{EPmb}) takes the form 
\begin{equation}\label{onedim}
\lm - u''+\frac{1}{m}|u'|^m-\be f=0\quad \text{in }\ \R,\qquad u(0)=0.    
\end{equation}
We first prove claim (ii) of Theorem \ref{main2}. 

\begin{proof}[Proof of Theorem \ref{main2} (ii)]
We may assume without loss of generality that $f\in C^{0,1}(\R^N)$. We prove that $\lm_{m,\be}<0$ for any $\be>0$.
We argue by contradiction assuming that $\lm_{m,\be}=0$ for some $\be>0$. Let $C>0$ be such that  $C^m=\max_{\R^N}(\be f)_-$, and let $u\in C^{2}(\R)$ be a classical solution of (\ref{onedim}) with $\lm=0$. Then, we see that
\begin{equation*}
-u''+\frac{1}{m}|u'|^m=\be f\geq -C^m\quad \text{in }\ \R.
\end{equation*}
By changing the variable such as $s=u'(x)/C$, we have 
\begin{equation*}
C^{m-1}x\geq \int_{u'(0)/C}^{u'(x)/C}\frac{m}{|s|^m+m}\,ds\geq -\int_\R\frac{m}{|s|^m+m}\,ds>-\infty\quad \text{for all }x\in\ \R.
\end{equation*}
Sending $x\to -\infty$, we get a contradiction. Hence, $\lambda_{m,\be}<0$ for all $\be>0$.
\end{proof}

We finally prove claim (iii) of Theorem \ref{main2}.
Let $N=1$ and $m=\infty$. In this case, (\ref{EPinfb}) can be written as
\begin{equation}\label{EP3}
\max\{\lm-u''-\be f,|u'|-1\}=0\quad \text{in }\ \R,\qquad u(0)=0.
\end{equation}

\begin{prop}\label{L}
Let $N=1$ and $m=\infty$. Suppose that $f_-\not\equiv0$, and set 
\begin{equation*}
L:=\int_\R f_-(u)du,\qquad
K:=\sup\Big\{\int_x^y -f(u)du\,\Big|\,x,y\in\R,\,x<y\Big\}.
\end{equation*}
Then $2/L\leq \be_+\leq 2/K$, where $2/L:=0$ for $L=\infty$ and $2/K:=0$ for $K=\infty$. 

\end{prop}

\begin{proof}
We first show that $2/L\leq \be_+$. We may assume $L<\infty$, otherwise the inequality is obvious. Notice here that $L>0$ by assumption. We set $\be_0:=2/L$ and construct a classical subsolution $u\in C^{2}(\R)$ of (\ref{EP3}) with $\lm=0$ and $\be=\be_0$.  
Let us consider the linear equation
\begin{equation}\label{lin}
-u'' +\be_0 f_-=0\quad \text{in }\R,\qquad u(0)=0.
\end{equation}
Then, for any $C\in\R$, the function $u\in C^2(\R)$ defined by
\begin{equation}\label{repu}
u(x)=\be_0 \int_{0}^x F(y)dy+Cx,\quad F(y):=\int_{0}^y f_-(u)du,
\end{equation} 
is a classical solution to (\ref{lin}). We now choose 
\begin{equation*}
C:=\frac1L\left(\int_{-\infty}^0 f_-(u)du-\int_0^{\infty}f_-(u)du\right).
\end{equation*}
Then, noting that $u'(x)=\be_0 F(x)+C$ for all $x\in \R$, we have
\begin{align*}
u'(x)\leq \frac2L\int_{0}^\infty f_-(u)du+C=1,\quad 
u'(x)\geq -\frac2L\int_{-\infty}^0 f_-(u)du+C=-1
\end{align*}
for all $x\in\R$. Hence, $u$ with the above $C$ is a subsolution of (\ref{EP3}) with $\lm=0$ and $\be=\be_0$, which implies that $\be_+ \geq 2/L$. 

We next show that $\be_+\leq 2/K$. Recall that $K>0$ by assumption. We argue by contradiction assuming that $\be_+> 2/K$. Fix any $\be$ such that $2/K<\be<\be_+$. Then, $\lm_{\infty,\be}=0$ by the definition of $\be_+$. Fix an arbitrary $\de>0$. Then, in view of Proposition \ref{C11}, there exists a classical subsolution $u\in C^{\infty}(\R)$ of (\ref{EP3}) with $\lm=-\de$.  In particular, we have
\begin{equation*}
-\de-u''-\be f\leq 0,\quad |u'|\leq 1\quad \text{in \ }\R.
\end{equation*}
This yields that, for any $x,y\in\R$ with $x<y$, 
\begin{align*}
\be \int_x^y-f(s)ds
\leq \int_x^y(u''(s)+\de)ds
=u'(y)-u'(x)+\de(y-x)
\leq 2+\de(y-x).
\end{align*}
Letting $\de\to 0$ and taking the supremum over all $x,y\in\R$ such that $x<y$, we obtain $\be K\leq 2$, which is a contradiction. Hence, we have completed the proof. 
\end{proof}

Claim (iii) of Theorem \ref{main2} is a direct consequence of the above proposition.

\begin{rem}\label{fminus}
Suppose that $f_+\equiv0$, that is, $f\leq 0$ in $\R$. Then $L=K=\int_\R |f(u)|du$, so that $\be_+=2/\int_\R |f(u)|du$. This implies that $\be_+>0$ if and only if $f\in L^1(\R)$.
\end{rem}

\begin{rem}
As far as the uniqueness for $u$, up to an additive constant, is concerned, equation (\ref{EPinf.intro}) with $\lm=\lm_{\infty}$ may have multiple solutions  in general. Indeed, let $N=1$ and $f(x):=-(1-|x|)_+$ in (\ref{EPinf.intro}). Then, in view of Remark \ref{fminus}, it is not difficult to observe that $\lm_{\infty}=0$. Furthermore, we define $u:\R\to\R$ by
\begin{equation*}
u(x):=\int_{0}^x F(y)dy+Cx,\quad F(y):=\int_{0}^y (1-|u|)_+du,
\end{equation*} 
where $C\in \R$ is a constant. Then, similarly as in the proof of Proposition \ref{L}, we see that $u$ is a classical solution of (\ref{EPinf.intro}) for any $C\in [-1/2,1/2]$. In particular, uniqueness for $u$ does not hold without any growth condition as $|x|\to\infty$. We remark here that, if $N=1$ and $f$ is convex and superlinear with respect to $x$, then, up to an additive constant, there exists only one viscosity solution $u$ of (\ref{EPinf.intro}) which satisfies $u(x)/|x|\to1$ as $|x|\to\infty$ (see \cite[Proposition 5.1]{Hy15}). At this stage, we do not know any uniqueness result for (\ref{EPinf.intro}) under our assumptions (A1)-(A2). 
\end{rem}

% ------------------------------------------------------------------------

\subsection*{Acknowledgment}
The first author's research was partially supported by Spanish grant MTM2011-25287. The second author's research was partially supported  by JSPS KAKENHI Grant Number 15K04935.

% ------------------------------------------------------------------------
\end{document}